\documentclass[a4paper,11pt,reqno]{amsart}

\usepackage{amsmath,amssymb,eucal,amsthm}
\usepackage{latexsym,bm}
\usepackage{mathrsfs,amsmath,amssymb,secdot}
\usepackage{amscd}
\def\N{\mathbb{N}}
\def\Z{\mathbb{Z}}
\def\Q{\mathbb{Q}}

\def\C{\mathbb{C}}

\def\ccc{{\mathfrak{C}}}

\def\gg{{\widetilde{G}}}

\allowdisplaybreaks
\numberwithin{equation}{section}

\newtheorem{theorem}{Theorem}[section]
\newtheorem{lemma}[theorem]{Lemma}

\theoremstyle{definition}

\newtheorem{example}[theorem]{Example}

\theoremstyle{remark}

\DeclareFontEncoding{OT2}{}{}
\DeclareFontSubstitution{OT2}{cmr}{m}{n}
\DeclareSymbolFont{cyss}{OT2}{wncyss}{m}{n}
\DeclareMathSymbol{\sh}{\mathbin}{cyss}{`x}

\numberwithin{equation}{section}

 \numberwithin{equation}{section}

\allowdisplaybreaks

\begin{document}

\title{On multiple series of Eisenstein type}




\author{Henrik Bachmann}


\address{H. Bachmann, 
              Department of Mathematics, Universit\"at Hamburg, Bundesstrasse 55, 20146 Hamburg, Germany}
\email{henrik.bachmann@math.uni-hamburg.de}

\author{Hirofumi Tsumura
}


\address{
H. Tsumura, 
              Department of Mathematics and Information Sciences, 
Tokyo Metropolitan University, 
1-1, Minami-Ohsawa, Hachioji, 
Tokyo 192-0397, Japan}
              \email{tsumura@tmu.ac.jp}           


\maketitle

\begin{abstract}
  The aim of this paper is to study certain multiple series which can be regarded as multiple analogues of Eisenstein series. As a prior research, the second-named author considered double analogues of Eisenstein series and expressed them as polynomials in terms of ordinary Eisenstein series. This fact was derived from the analytic observation of infinite series involving hyperbolic functions which were based on the study of Cauchy, and also Ramanujan. In this paper, we prove an explicit relation formula among these series. This gives an alternative proof of this fact by using the technique of partial fraction decompositions of multiple series which was introduced by Gangl, Kaneko and Zagier. By the same method, we further show a certain multiple analogue of this fact and give some examples of explicit formulas. Finally we give several remarks about the relation between our present result and the previous work for infinite series involving hyperbolic functions.  


\end{abstract}

\section{Introduction}\label{sec-1}

Let $\mathbb{N}$, $\mathbb{Z}$, $\mathbb{Q}$, $\mathbb{R}$, $\mathbb{C}$ be the sets of natural numbers, rational integers,  rational numbers,  real numbers and complex numbers, respectively. Throughout this paper, the empty sum and the empty product are interpreted as $0$ and $1$, respectively.

Let $G_{2j}(\tau)$ $(j\in \mathbb{N}_{\geq 2})$ be the ordinary 
Eisenstein series defined by
\begin{equation}
G_{2j}(\tau)=\sum_{m\in \mathbb{Z}}\sum_{n\in \mathbb{Z}\atop (m,n)\not=(0,0)}\frac{1}{(m+n\tau)^{2j}} \label{Eisenstein}
\end{equation}
for $\tau \in \mathbb{C}$ with ${\rm Im} (\tau)>0$ (see, for example, Koblitz\,\cite{Ko} and Serre\,\cite{Se}). Further $G_{2}(\tau)$ can be defined by 
\begin{equation}
G_{2}(\tau)=\sum_{m\in \mathbb{Z}}\sum_{n\in \mathbb{Z}\atop (m,n)\not=(0,0)}\frac{1}{(m+n\tau)^{2}}, \label{Eisenstein-2}
\end{equation}
which is conditionally convergent with respect to the order of summation as above. It should be noted that the usual definition of $G_2(\tau)$ is 
\begin{equation*}
\sum_{m\in \mathbb{Z}}\sum_{n\in \mathbb{Z}\atop (m,n)\not=(0,0)}\frac{1}{(m\tau+n)^{2}}, 
\end{equation*}
which is equal to $G_2(\tau)+2\pi i/\tau$ using \eqref{Eisenstein-2} (see \cite[Section 7.4]{Se}). We here adopt the definition \eqref{Eisenstein-2} which is suitable from the viewpoint of the connection with certain Eisenstein type series involving hyperbolic functions studied in \cite{TsB,TsAus} (see below). 

In \cite{gkz}, 
Gangl, Kaneko and Zagier defined the double Eisenstein series 
\begin{equation}
G_{r,s}(\tau)=\sum_{{\bf n_1},{\bf n_2}\in \mathbb{Z}+\mathbb{Z}\tau \atop {\bf n_1} \succ {\bf n_2} \succ 0}
\frac{1}{{\bf n_1}^r{\bf n_2}^s}\quad (r\in \mathbb{N}_{\geq 3},\ s\in \mathbb{N}_{\geq 2}),   \label{GKZ-1}
\end{equation}
where ${\bf n} \succ 0$ means ${\bf n}=m+n\tau$ with $n>0$ or $n=0$, $m>0$, and ${\bf m} \succ {\bf n}$ means ${\bf m}-{\bf n} \succ 0$. 
They gave a new approach to investigation of multiple zeta values by 
considering the Fourier expansion of \eqref{GKZ-1}, which was shown by the partial fraction decomposition (see \cite[Theorem 6]{gkz}). 

In \cite{TsPJM}, the second-named author 
considered another type of the double Eisenstein series defined by
\begin{align}
\gg_{2p,2q}(\tau)& =\sum_{m\in \mathbb{Z}} \sum_{n_1 \in \mathbb{Z}\atop (m,n_1)\not=(0,0)}\sum_{n_2\in \mathbb{Z}\atop (m,n_2)\not=(0,0)}\frac{1}{(m+n_1\tau)^{2p}(m+n_2\tau)^{2q}}\quad (p,q \in \mathbb{N}) \label{e-1-1}
\end{align}
and its level-$N$ version based on the previous results \cite{TsB,TsAus}. We can easily see that, except for the case $(p,q)=(1,1)$, the right-hand side of \eqref{e-1-1} converges absolutely. Hence the order of summation can be changed, namely $\gg_{2p,2q}(\tau)=\gg_{2q,2p}(\tau)$. In the case $(p,q)=(1,1)$, the right-hand side of \eqref{e-1-1} is conditionally convergent with respect to the order of summation as above. It was shown that
\begin{align}
& \tau^{2(p+q)}\widetilde{G}_{2p,2q}(\tau)\in \mathbb{Q}\left[ \tau,\, \pi,\,G_{2}(\tau),\,G_4(\tau),\,G_6(\tau)\right] \label{e-1-2}
\end{align}
for $p,q\in \mathbb{N}$ (see \cite[Theorem 4.2]{TsPJM}). For example, we have
\begin{align}
 \gg_{2,2}(\tau)&=G_4(\tau)+\frac{2\pi^2}{3\tau^2}G_2(\tau)-\frac{2\pi^4}{15\tau^4}, \label{G22-t}\\
 \gg_{2,4}(\tau)&=G_6(\tau)+\frac{\pi^2}{3\tau^2}G_4(\tau)+\frac{4\pi^4}{45\tau^4}G_2(\tau)-\frac{2\pi^6}{63\tau^6}. \label{G24-t}
\end{align}

Note that 
Pasles and Pribitkin \cite{Pas} gave a generalization of the Lipschitz summation formula which seems to be related to $\widetilde{G}_{2p,2q}(\tau)$, 
while it is not yet clear. 

The motivation to study $\gg_{2p,2q}(\tau)$ occurs from interest to the double series of Eisenstein type defined by
\begin{equation}\label{coth-Eisen}
\ccc_{j}^{k}(\tau):=\sum_{m\in \mathbb{Z}\atop m\ne 0} \sum_{n \in \mathbb{Z}}\frac{\coth^{k} ((m+n\tau)\pi i/\tau)}{(m+n\tau)^{j}}\quad (j\in \mathbb{N}_{\geq 2}).
\end{equation}
As is well-known, Cauchy proved, and Ramanujan rediscovered the interesting formulas
\begin{align}\label{Cauchy}
&\sum_{m\in \mathbb{Z}\atop m\ne 0} \frac{\coth (m\pi)}{m^{4p+3}}=(2\pi)^{4p+3}\sum_{\nu=0}^{2p+2}(-1)^{\nu+1} \frac{B_{2\nu}}{(2\nu)!}\frac{B_{4p+4-2\nu}}{(4p+4-2\nu)!} \quad (p\in \mathbb{Z}_{\geq 0}),
\end{align}
where $\{B_n\}$ are the Bernoulli numbers (see Cauchy \cite[pp.\,320,\,361]{Ca}, also Ramanujan's notebooks \cite[p.\,293,\,(25.3)]{Be2}). As Eisenstein-type analogues of \eqref{Cauchy}, the second-named author proved
\begin{equation}\label{coth-Eisen-2}
\ccc_{j}^{k}(i)=\sum_{m\in \mathbb{Z}\atop m\ne 0} \sum_{n \in \mathbb{Z}}\frac{\coth^{k} ((m+ni)\pi)}{(m+ni)^{j}} \in \mathbb{Q}\left[\frac{1}{\pi},\pi,\varpi\right]
\end{equation}
for $j,k \in \mathbb{N}$ with $j\geq 2$ and $j \equiv k\ (\textrm{mod}\ 2)$, where $i=\sqrt{-1}$ and $\varpi$ is the lemniscate constant defined by \eqref{lemniscate} (see \cite[Cororally 2.5]{TsAus}). Furthermore its generalization for $\tau$ was given (see \cite[Theorem 2.3]{TsPJM}). It is noted that $\gg_{2p,2q}(\tau)$ is closely connected with \eqref{coth-Eisen}. Actually the result for \eqref{coth-Eisen} stated above yields \eqref{e-1-2}.

In this paper, we first give an alternative and more simple proof of \eqref{e-1-2} using the partial fraction decomposition similar to that in \cite{gkz}. More explicitly, we give the following expression formulas of $\gg_{2p,2q}(\tau)$ in terms of ordinary Eisenstein series, where $\zeta(s)$ is the Riemann zeta-function.

\begin{theorem}\label{thm1}
For $p,q \in \N$,
\begin{align*}
\widetilde{G}_{2p,2q}(\tau)& = G_{2(p+q)}(\tau) \\
& \quad + \sum_{l_1,l_2\in \mathbb{N}\atop l_1 + l_2 = p+q} \left( \frac{2 \zeta(2l_2) \binom{2l_2-1}{2q-1} }{\tau^{2l_2}} G_{2l_1}(\tau) + \frac{2 \zeta(2l_1) \binom{2l_1-1}{2p-1} }{\tau^{2l_1}}  G_{2l_2}(\tau) \right) \\
& \quad  + \frac{4}{\tau^{2(p+q)}}\bigg[ \zeta(2p) \zeta(2q) - \frac{1}{2} \zeta(2(p+q))\\
& \qquad \qquad \qquad - \sum_{l_1,l_2\in \mathbb{N} \atop l_1 + l_2 = p+q}\zeta(2l_1)\zeta(2l_2) \left( \binom{2l_2-1}{2q-1} + \binom{2l_1-1}{2p-1} \right) \bigg].
\end{align*}
In particular, $\tau^{2(p+q)} \widetilde{G}_{2p,2q}(\tau) \in \Q \left[ \tau^2, \pi^2, G_2(\tau), G_4(\tau), G_6(\tau) \right]$.
\end{theorem}

Furthermore we consider the multiple series analogue of \eqref{e-1-1} defined by
\begin{align}
\gg_{2p_1,\ldots,2p_r}(\tau)& =\sum_{m\in \mathbb{Z}} \sum_{n_1 \in \mathbb{Z}\atop (m,n_1)\not=(0,0)}\cdots \sum_{n_r\in \mathbb{Z}\atop (m,n_r)\not=(0,0)}\prod_{j=1}^{r}\frac{1}{(m+n_j\tau)^{2p_j}}\label{e-1-4}
\end{align}
for $r\in \mathbb{N}_{\geq 2}$ and $p_1,\ldots,p_r\in \mathbb{N}$. Similar to the above notice, except for the case $(p_1,\ldots,p_r)=(1,\ldots,1)$, the right-hand side of \eqref{e-1-4} converges absolutely. Hence the order of summation can be changed. In the case $(p_1,\ldots,p_r)=(1,\ldots,1)$, the right-hand side of \eqref{e-1-4} is conditionally convergent with respect to the order of summation as above. 

With the above notation, we prove the general multiple version of \eqref{e-1-2} as follows.

\begin{theorem}\label{T-1}\ \ For $r\in \mathbb{N}_{\geq 2}$ and $p_1,\ldots,p_r\in \mathbb{N}$, 
\begin{align}
& \tau^{2(p_1+\cdots+p_r)}\widetilde{G}_{2p_1,\ldots,2p_r}(\tau)\in \mathbb{Q}\left[ \tau^2,\, \pi^2,\,G_{2}(\tau),\,G_4(\tau),\,G_6(\tau)\right]. \label{e-1-5}
\end{align}
\end{theorem}

For example,
$$ \gg_{2,2,2}(\tau)=G_6(\tau)+\frac{\pi^2}{\tau^2}G_4(\tau)+\frac{8\pi^4}{15\tau^4}G_2(\tau)-\frac{52\pi^6}{315\tau^6}.$$

\ 

We give the proof of Theorem \ref{thm1} in Section \ref{sec-2}, and that of Theorem \ref{T-1} in Section \ref{sec-3}. In Section \ref{sec-4}, we briefly remark a relation between our present result and the study of multiple series involving hyperbolic functions defined by
$$\sum_{m\in \mathbb{Z}\atop m\ne 0} \sum_{n_1 \in \mathbb{Z}}\cdots \sum_{n_r\in \mathbb{Z}}{\coth^{2k} ((m+n_r\tau)\pi i/\tau)}\prod_{j=1}^{r}\frac{1}{(m+n_j\tau)^{2p_j}},$$
which is a multiple version of \eqref{coth-Eisen}. In fact, we can show a certain multiple analogue of \eqref{coth-Eisen-2}.

Finally we note that this paper contains a lot of examples of formulas which were numerically checked by using Mathematica 9. 

\

\section{Proof of Theorem \ref{thm1}}\label{sec-2}

In this section, we give the proof of Theorem \ref{thm1}. 
First we prepare the following two series for $p,q\in \mathbb{N}$:
\begin{align*} 
& \widetilde{G}^*_{2p}(\tau) := \sum_{\substack{m \in \Z \\ m \neq 0}} \sum_{n_1 \in \Z} \frac{1}{(m+n_1 \tau)^{2p}} \,,\\
& \widetilde{G}^*_{2p,2q}(\tau) := \sum_{\substack{m \in \Z \\ m \neq 0}} \sum_{n_1\in \Z}\sum_{n_2 \in \Z} \frac{1}{(m+n_1 \tau)^{2p} (m+n_2\tau)^{2q}},  
\end{align*}
where the order of summation is fixed in the case of $p=q=1$. 
It is easy to check that 
\begin{equation} \label{G_Gstar}
\widetilde{G}_{2p}(\tau)  = \widetilde{G}^*_{2p}(\tau)  +\frac{2 \zeta(2p) }{\tau^{2p}} \,, \quad \widetilde{G}_{2p,2q}(\tau) = \widetilde{G}^*_{2p,2q}(\tau) + \frac{4 \zeta(2q) \zeta(2q)}{\tau^{2(p+q)}} \,.
\end{equation}

To calculate $\widetilde{G}^*$, we use an approach similar to that in the consideration of the Fourier expansion of \eqref{GKZ-1} introduced in \cite{gkz}. Actually we use the partial fraction expansion 
\begin{equation} \label{partfrac_len2}
\begin{split}
& \frac{1}{(x+c_1)^{s_1} (x+c_2)^{s_2}} \\
& \ = \sum_{k_1,k_2\in \mathbb{N} \atop k_1 + k_2 = s_1+s_2} \left( \frac{(-1)^{s_2} \binom{k_2-1}{s_2-1} }{ (c_1-c_2)^{k_2} (x+c_1)^{k_1} }   + \frac{(-1)^{s_1} \binom{k_1-1}{s_1-1} }{ (c_2-c_1)^{k_1}(x+c_2)^{k_2} }  \right) \,
\end{split}
\end{equation}
which is valid for $c_1,c_2 \in \C$ with $c_1 \neq c_2$ and $s_1,s_2 \in \N$ (see, for example, \cite[Section 2]{gkz}). Splitting up the summation of $\widetilde{G}^*_{2p,2q}$ into the parts where $n_1=n_2$ and $n_1 \neq n_2$ together with $c_j = \tau n_j,\, x=m$ and $s_1= 2p$, $s_2 = 2q$ in \eqref{partfrac_len2}, we obtain
\begin{align*}
\widetilde{G}^*_{2p,2q}(\tau) &=  \left\{\sum_{\substack{m \in \Z \\ m \neq 0}} \sum_{n_1=n_2 \in \Z} +  \sum_{\substack{m \in \Z \\ m \neq 0}} \sum_{\substack{n_1, n_2 \in \Z\\ n_1 \neq n_2}}\right\} \frac{1}{(m+n_1 \tau)^{2p} (m+n_2\tau)^{2q}} \\
&= \widetilde{G}^*_{2(p+q)}(\tau) + \sum_{\substack{m \in \Z \\ m \neq 0}} \sum_{\substack{n_1, n_2 \in \Z\\ n_1 \neq n_2}} \sum_{k_1,k_2\in \mathbb{N} \atop k_1 + k_2 = 2(p+q)} \left( \frac{1}{\tau^{k_2}} \frac{\binom{k_2-1}{2q-1} }{ (n_1-n_2)^{k_2} (m+n_1 \tau)^{k_1} } \right) \\
& \quad + \sum_{\substack{m \in \Z \\ m \neq 0}} \sum_{\substack{n_1, n_2 \in \Z\\ n_1 \neq n_2}} \sum_{k_1,k_2\in \mathbb{N} \atop k_1 + k_2 = 2(p+q)} \left( \frac{1}{\tau^{k_1}} \frac{ \binom{k_1-1}{2p-1} }{ (n_2-n_1)^{k_1} (m+n_2 \tau)^{k_2} } \right).
\end{align*}
It is easy to see that
\begin{equation} \label{sumtoprod}
 \frac{1}{\tau^{k_2}} \sum_{\substack{n_1, n_2 \in \Z\\ n_1 \neq n_2}}  \frac{\binom{k_2-1}{2q-1} }{ (n_1-n_2)^{k_2} (m+n_1 \tau)^{k_1}} = \frac{2 \zeta^\dagger (k_2) \binom{k_2-1}{2q-1} }{\tau^{k_2}}  \sum_{\substack{n_1 \in \Z}} \frac{1}{ (m+n_1 \tau)^{k_1}} \,, 
 \end{equation}
where we set
\begin{equation*}
\zeta^\dagger (k) = 
\begin{cases} 
\zeta(k) & \text{(if $k$ is even)}\\
0 &\text{(if $k$ is odd)}.
\end{cases}
\end{equation*}
Therefore we can write
\begin{align*}
 \widetilde{G}^*_{2p,2q}(\tau)  &= \widetilde{G}^*_{2(p+q)}(\tau) \\
& \quad + \sum_{k_1,k_2\in \mathbb{N}\text{:even} \atop k_1 + k_2 = 2(p+q)}\left( \frac{2 \zeta^\dagger (k_2) \binom{k_2-1}{2q-1} }{\tau^{k_2}}  G^*_{k_1}(\tau) + \frac{2 \zeta^\dagger (k_1) \binom{k_1-1}{2p-1} }{\tau^{k_1}}  G^*_{k_2}(\tau)\right) \,.
\end{align*}
Thus it follows from \eqref{G_Gstar} that the former assertion of the theorem holds. The latter assertion follows from the well-known facts $G_{2k} \in \Q[G_4,G_6]$ for $k\in \mathbb{N}_{\geq 2}$ and $\zeta(2l) \in \Q[\pi^{2}]$ for $l\in \mathbb{N}$ (see \cite{Ko,Se}). This completes the proof of Theorem \ref{thm1}.

\begin{example}
\begin{align}
 \gg_{2,6}(\tau)&=G_8(\tau) + \frac{\pi^2}{3\tau^2}G_6(\tau) + \frac{\pi^4}{15\tau^4}G_4(\tau) + \frac{4\pi^6}{315\tau^6}G_2(\tau) - \frac{4\pi^8}{675\tau^8}, \label{G26-t}\\
 \gg_{2,8}(\tau)&=G_{10}(\tau) + \frac{\pi^2}{3\tau^2}G_8(\tau) + \frac{\pi^4}{15\tau^4}G_6(\tau) + \frac{2\pi^6}{189\tau^6}G_4(\tau) \label{G28-t}\\
& \quad + \frac{8\pi^8}{4725\tau^8}G_2(\tau) - \frac{2\pi^{10}}{2079\tau^{10}},\notag
\end{align}
where we note that 
\begin{equation}
G_8(\tau)=\frac{3}{7}G_4(\tau)^2,\quad G_{10}(\tau)=\frac{5}{11}G_4(\tau)G_6(\tau) \label{Eisen-rel}
\end{equation}
(see \cite[Section 7]{Se}). In particular when $\tau=i=\sqrt{-1}$, we make use of the well-known Hurwitz formulas:
\begin{equation}
G_2(i)=-\pi,\ G_4(i)=\frac{\varpi^4}{15},\ G_6(i)=0,\ G_8(i)=\frac{\varpi^8}{525},\ldots \label{e-1-3}
\end{equation}
(see, for example, \cite[Section 8]{Le}), where $\varpi$ is the lemniscate constant defined by 
\begin{equation}\label{lemniscate}
\varpi=2\int_{0}^{1}\frac{dx}{\sqrt{1-x^4}}=\frac{\Gamma(1/4)^2}{2\sqrt{2\pi}}=2.6220575542921\cdots.
\end{equation}
By these results, we have from \eqref{G22-t}-\eqref{G28-t} that
\begin{align*}
& \widetilde{G}_{2,2}(i)=\frac{\varpi^4}{15}-\frac{2}{15}\pi^4+\frac{2}{3}\pi^3, \\
& \widetilde{G}_{2,4}(i)=-\frac{\varpi^4\pi^2}{45}+\frac{2}{63}\pi^6-\frac{4}{45}\pi^5, \\
& \widetilde{G}_{2,6}(i)=\frac{\varpi^8}{525}+\frac{\varpi^4\pi^4}{225}-\frac{4}{675}\pi^8+\frac{4}{315}\pi^7,\\ 
& \widetilde{G}_{2,8}(i)=- \frac{\varpi^8\pi^2}{1575}- \frac{2}{2835}\varpi^4\pi^6 +\frac{2}{2079}\pi^{10}- \frac{8}{4725}\pi^9, 
\end{align*}
which understandably coincide with the results in \cite[Example 4.3]{TsPJM}. 
\end{example}

\ 

\section{Proof of Theorem \ref{T-1}}\label{sec-3}

Similar to Theorem \ref{thm1}, we can give the proof of Theorem \ref{T-1} by induction on $r\geq 2$. In fact, the case of $r=2$ is nothing but Theorem \ref{thm1}. 
Clearly one can again focus on 
\begin{equation}
\widetilde{G}^*_{2p_1,\dots,2p_r}(\tau) = \sum_{\substack{m \in \Z \\ m \neq 0}} \sum_{n_1 \in \Z}\cdots \sum_{n_r \in \Z} \prod_{j=1}^r \frac{1}{(m+n_j \tau)^{2p_j}}.  \label{G-ast}
\end{equation}
We immediately see that
\begin{equation}
\widetilde{G}_{2p_1,\dots,2p_r}(\tau) = \widetilde{G}^*_{2p_1,\dots,2p_r}(\tau) +\prod_{j=1}^{r}\left(\frac{2\zeta(2p_j)}{\tau^{2p_j}}\right). \label{G-G^ast}
\end{equation}
Similar to the argument in Section \ref{sec-2}, we can prove the following.

\begin{lemma}\label{G-induction}
\ For $r\in \mathbb{N}_{\geq 2}$ and $p_1,\ldots,p_r\in \N$, 
\begin{align}
& \widetilde{G}^*_{2p_1,\ldots,2p_r}(\tau) \notag\\
& =\widetilde{G}^*_{2p_1,\ldots,2p_{r-2},2(p_{r-1}+p_r)}(\tau) + \sum_{l_1,l_2\in \mathbb{N} \atop l_1 + l_2 = p_{r-1}+p_r}\frac{2 \zeta(2l_2) \binom{2l_2-1}{2p_r-1} }{\tau^{2l_2}}  \widetilde{G}^*_{2p_1,\ldots,2p_{r-2},2l_1}(\tau) \notag\\
& \quad + \sum_{l_1,l_2\in \mathbb{N} \atop l_1 + l_2 = p_{r-1}+p_r}\frac{2 \zeta(2l_1) \binom{2l_1-1}{2p_{r-1}-1} }{\tau^{2l_1}}  \widetilde{G}^*_{p_1,\ldots,p_{r-2},2l_2}(\tau). \label{G^ast-expr}
\end{align}
\end{lemma}

\begin{proof}
Similar to the proof of Theorem \ref{thm1}, for $r\geq 2$, we have
\begin{align*}
&\widetilde{G}^*_{2p_1,\ldots,2p_r}(\tau) =  \left\{\sum_{\substack{m \in \Z \\ m \neq 0}} \sum_{\substack{n_1,\ldots,n_r \in \Z \\ n_{r-1}=n_r}} +  \sum_{\substack{m \in \Z \\ m \neq 0}} \sum_{\substack{n_1, \ldots,n_{r} \in \Z\\ n_{r-1}\neq n_r}}\right\} \prod_{j=1}^{r}\frac{1}{(m+n_j \tau)^{2p_j}} \\
&\ = \widetilde{G}^*_{2p_1,\ldots,2p_{r-2},2(p_{r-1}+p_r)}(\tau) \\
&  \quad + \sum_{\substack{m \in \Z \\ m \neq 0}} \sum_{\substack{n_1,\ldots, n_r \in \Z\\ n_{r-1} \neq n_r}} \sum_{k_1,k_2\in \mathbb{N} \atop k_1 + k_2 = 2(p_{r-1}+p_r)} \left( \frac{1}{\tau^{k_2}} \frac{\binom{k_2-1}{2p_r-1} }{ (n_{r-1}-n_r)^{k_2} (m+n_{r-1} \tau)^{k_1} } \right) \\
& \quad + \sum_{\substack{m \in \Z \\ m \neq 0}} \sum_{\substack{n_1,\ldots, n_r \in \Z\\ n_{r-1} \neq n_r}} \sum_{k_1,k_2\in \mathbb{N} \atop k_1 + k_2 = 2(p_{r-1}+p_r)} \left( \frac{1}{\tau^{k_1}} \frac{ \binom{k_1-1}{2p_{r-1}-1} }{ (n_{r}-n_{r-1})^{k_1} (m+n_r \tau)^{k_2} } \right).
\end{align*}
By using \eqref{sumtoprod}, we obtain \eqref{G^ast-expr}.  
\end{proof}

\begin{proof}[of Theorem \ref{T-1}] 
It follows from Lemma \ref{G-induction} and \eqref{G-ast}, we obtain the assertion of Theorem \ref{T-1} induction on $r\geq 2$. 
\end{proof}

It would also be possible, but tedious, to give explicit formulas for $\widetilde{G}_{2p_1,\dots,2p_r}(\tau)$ in terms of only $G_{2p}(\tau)$ and $\zeta(s)$, similar to the case $r=2$. As inductive computations, the above argument gives an algorithm to express $\widetilde{G}_{2p_1,\dots,2p_r}(\tau)$ in terms of $G_{2p}(\tau)$ and $\zeta(s)$.

\begin{example}\label{Ex-3-1}\ \ Using the algorithm introduced in the proof of Theorem \ref{T-1}, we obtain the following formulas:
\begin{align*}
& \gg_{2,2,2}(\tau)=G_6(\tau)+\frac{\pi^2}{\tau^2}G_4(\tau)+\frac{8\pi^4}{15\tau^4}G_2(\tau)-\frac{52\pi^6}{315\tau^6},\\
& \gg_{2,4,2}(\tau)\left(= \gg_{2,2,4}(\tau)\right)\\
& \quad =G_8(\tau)+ \frac{2\pi^2}{3\tau^2}G_6(\tau)+ \frac{4\pi^4}{15\tau^4}G_4(\tau)+ \frac{32\pi^6}{315\tau^6}G_2(\tau)-\frac{184\pi^8}{4725\tau^8}, \\
& \gg_{2,6,2}(\tau)\left(= \gg_{2,2,6}(\tau)\right)\\
& \quad =G_{10}(\tau)+ \frac{2\pi^2}{3\tau^2}G_8(\tau)+ \frac{11\pi^4}{45\tau^4}G_6(\tau)+ \frac{64\pi^6}{945\tau^6}G_4(\tau)+ \frac{32\pi^8}{1575\tau^8}G_2(\tau)-\frac{272\pi^{10}}{31185\tau^{10}},
\end{align*}
where we note \eqref{Eisen-rel}. Letting $\tau=i$, we obtain 
\begin{align*}
& \gg_{2,2,2}(i)=-\frac{\varpi^4\pi^2}{15}+\frac{52\pi^6}{315}-\frac{8\pi^5}{15},\\
& \gg_{2,4,2}(i)=\gg_{2,2,4}(i)=\frac{\varpi^8}{525}+ \frac{4\varpi^4\pi^4}{225} -\frac{184\pi^8}{4725} + \frac{32\pi^7}{315},\\
& \gg_{2,6,2}(i)=\gg_{2,2,6}(i)=  - \frac{2\varpi^8\pi^2}{1575}- \frac{64\varpi^4\pi^6}{14175}+\frac{272\pi^{10}}{31185} - \frac{32\pi^9}{1575}. 
\end{align*}
Also, letting $\tau=\rho$, we obtain
$$\gg_{2,2,2}(\rho)=\frac{\widetilde{\varpi}^6}{35}-\frac{52\pi^6}{315}+\frac{16\pi^5}{15\sqrt{3}}.$$
Furthermore, using the above result, we can give
\begin{align*}
\gg_{2,2,2,2}(\tau)&=\sum_{m\in \mathbb{Z}} \sum_{n_1,n_2,n_3,n_4 \in \mathbb{Z}\atop \{(m,n_j)\not=(0,0)\}_{j=1,2,3,4}}\frac{1}{(m+n_1\tau)^2(m+n_2\tau)^2(m+n_3\tau)^2(m+n_4\tau)^2}\\
& =G_8(\tau)+ \frac{4\pi^2}{3\tau^2}G_6(\tau)+ \frac{14\pi^4}{15\tau^4}G_4(\tau) + \frac{16\pi^6}{35\tau^6}G_2(\tau)-\frac{86\pi^8}{525\tau^8}.
\end{align*}
Therefore we obtain 
\begin{align}
\gg_{2,2,2,2}(i)&=\sum_{m\in \mathbb{Z}} \sum_{n_1,n_2,n_3,n_4 \in \mathbb{Z}\atop \{(m,n_j)\not=(0,0)\}_{j=1,2,3,4}}\frac{1}{(m+n_1i)^2(m+n_2i)^2(m+n_3i)^2(m+n_4i)^2}\label{G2222}\\
& =\frac{\varpi^8}{525} + \frac{14\varpi^4\pi^4}{225}-\frac{86\pi^8}{525} + \frac{16\pi^7}{35}. \notag
\end{align}

\end{example}

\ 

\section{Multiple series involving hyperbolic functions}\label{sec-4}

We briefly state a relation between the result in the preceding section and the result in \cite{TsPJM}. We here consider the multiple series
\begin{align*}
& \ccc_r^{\langle 2k \rangle}(2p_1,\ldots,2p_{r};\tau)=\sum_{m\in \mathbb{Z}\atop m\ne 0} \sum_{n_1 \in \mathbb{Z}}\cdots \sum_{n_r\in \mathbb{Z}}{\coth^{2k} ((m+n_r\tau)\pi i/\tau)}\prod_{j=1}^{r}\frac{1}{(m+n_j\tau)^{2p_j}} \notag
\end{align*}
for $k\in \mathbb{Z}_{\geq 0}$ and $p_1,\ldots,p_{r}\in \mathbb{N}$. The case $r=1$ was already considered and calculated in \cite[Section 2]{TsPJM} as noted in Section \ref{sec-1}. Using the same method as in \cite{TsPJM}, we can prove that 
\begin{equation*}
\begin{split}
& \pi^{2k}\tau^{2(p_1+\cdots+p_{r})}\ccc_r^{\langle 2k \rangle}(2p_1,\ldots,2p_{r};\tau) \in \mathbb{Q}\left[\tau^2,\, \pi^2,\,\left\{\gg_{2p_1,\ldots,2p_{r-1},2\nu}(\tau)\right\}_{1\leq \nu\leq 2p_r+k}\right]
\end{split}
\end{equation*}
and the recursion relation 
\begin{align*}
\sum_{\mu=0}^{k}\binom{k}{\mu}(-1)^{k-\mu}& \ccc_{r}^{\langle 2\mu \rangle}(2p_1,\ldots,2p_r;\tau)=\frac{2^{2k}}{(2k-1)!}\sum_{h=1}^{k}\alpha(k,{2h-1})\frac{(2h-1)!}{(2\pi i/\tau)^{2h}}\notag\\
& \  \times \left\{ \gg_{2p_1,\ldots,2p_r,2h}(\tau)-2^{r+1}\tau^{-2(p_1+\cdots+p_r+h)}\prod_{j=1}^{r}\zeta(2p_j)\cdot \zeta(2h)\right\},
\end{align*}
where $\alpha(n,{k})$ is defined by
\begin{equation*}
X \prod_{l=1}^{n-1}(X-l)(X+l)=\sum_{k=0}^{2n-1}\alpha(n,{k})X^{k}.
\end{equation*}
These two facts can be surely regarded as multiple analogues of Theorems 2.3 and 4.1 in \cite{TsPJM}. The proofs of them can be given in the same way as in \cite[Sections 3 and 4]{TsPJM} by induction on $r$, while they are complicated. Hence we omit them here, and only remark that, combining these results, we obtain the proof of Theorem \ref{T-1} by using 
$\coth^2 x=1+1/\sinh^2 x$ and 
\begin{align*}
 \frac{1}{(\sinh(m\pi i/\tau))^{2\nu}}&=\frac{2^{2\nu}}{(2\nu-1)!}\sum_{j=1}^{\nu}\alpha(\nu,{2j-1})\frac{(2j-1)!}{(2\pi i)^{2j}}\sum_{l\in \mathbb{Z}}\frac{1}{(-m/\tau+l)^{2j}} \notag\\
&\qquad =\frac{2^{2\nu}}{(2\nu-1)!}\sum_{j=1}^{\nu}\alpha(\nu,{2j-1})\frac{(2j-1)!}{(2\pi i/\tau)^{2j}}\sum_{l\in \mathbb{Z}}\frac{1}{(m+l\tau)^{2j}}
\end{align*}
(see \cite[(4.6)]{TsPJM}). 
This proof of Theorem \ref{T-1} is more complicated than our present proof stated in Section \ref{sec-3}. On the other side, as benefits of this method, we obtain explicit formulas for $\ccc_{r}^{\langle 2k \rangle}(2p_1,\ldots,2p_r;\tau)$, for example, 
\begin{align*}
& \ccc_2^{\langle 2 \rangle}(2,2;\tau) =\sum_{m\in \mathbb{Z}\atop m\ne 0} \sum_{n_1 \in \mathbb{Z}} \sum_{n_2\in \mathbb{Z}}\frac{\coth^2 ((m+n_2\tau)\pi i/\tau)}{(m+n_1\tau)^{2}(m+n_2\tau)^{2}}\\
& \quad =\frac{-40\pi^6 + 126\tau^2\pi^4G_2(\tau) - 945\tau^6G_6(\tau)}{945\tau^4\pi^2},\\
&\ccc_2^{\langle 2 \rangle}(2,4;\tau) =\sum_{m\in \mathbb{Z}\atop m\ne 0} \sum_{n_1 \in \mathbb{Z}} \sum_{n_2\in \mathbb{Z}}\frac{\coth^2 ((m+n_2\tau)\pi i/\tau)}{(m+n_1\tau)^{2}(m+n_2\tau)^{4}}\\
&\quad =\frac{32\pi^8 - 180\tau^2\pi^6G_2(\tau) + 945\tau^4\pi^4G_4(\tau) + 4725\tau^6\pi^2G_6(\tau) - 14175\tau^8G_8(\tau)}{14175\tau^6\pi^2},
\end{align*}
which are multiple analogues of the formulas in \cite[Example 2.5]{TsPJM}. 

\

\end{document}